\documentclass[a4paper]{article}

\usepackage{optidef}
\usepackage{amsthm}
\usepackage{authblk}
\usepackage{amssymb}
\usepackage{booktabs}
\usepackage[bookmarks=false]{hyperref}

\hypersetup{
	colorlinks=true, 
	breaklinks=true,
	urlcolor= blue, 
	linkcolor= blue, 
	citecolor=red, 
	pdftitle={},
	pdfauthor={},
}

\theoremstyle{thmstyleone}%
\newtheorem{theorem}{Theorem}
\theoremstyle{thmstyletwo}%

\theoremstyle{thmstylethree}%
\newtheorem{definition}{Definition}%

\newcommand*{\sphline}{%
\noalign{\vskip3pt}%
\hline%
\noalign{\vskip3pt}}

\begin{document}


\title{Strengthening SONC Relaxations with Constraints Derived from Variable Bounds}


\author[1]{Ksenia Bestuzheva}

\author[2]{Helena V\"olker}

\author[3]{Ambros Gleixner}

\affil[1]{Zuse Institute Berlin, Takustr.~7, 14195 Berlin, Germany\\
          \texttt{bestuzheva@zib.de}}

\affil[2]{Zuse Institute Berlin, Takustr.~7, 14195 Berlin, Germany\\
          \texttt{voelker@zib.de}}

\affil[3]{Zuse Institute Berlin and HTW Berlin, Germany\\
          \texttt{gleixner@zib.de}}

\maketitle

\begin{abstract}Nonnegativity certificates can be used to obtain tight dual bounds for polynomial
optimization problems.
Hierarchies of certificate-based relaxations ensure convergence to the global optimum, but higher
levels of such hierarchies can become very computationally expensive, and the well-known sums of
squares hierarchies scale poorly with the degree of the polynomials.
This has motivated research into alternative certificates and approaches to global optimization.
We consider sums of nonnegative circuit polynomials (SONC) certificates, which are well-suited for sparse
problems since the computational cost depends on the number of terms in the
polynomials and does not depend on the degrees of the polynomials.
We propose a method that guarantees that given finite variable domains, a SONC relaxation will
yield a finite dual bound.
This method opens up a new approach to utilizing variable bounds in SONC-based methods, which is
particularly crucial for integrating SONC relaxations into branch-and-bound algorithms.
We report on computational experiments with incorporating SONC relaxations into the spatial branch-and-bound algorithm of the mixed-integer
nonlinear programming framework SCIP.
Applying our strengthening method increases the number of instances where the SONC relaxation of
the root node yielded a finite dual bound from 9 to 330 out of 349
instances in the test set.
\end{abstract}

\section{Introduction}

Polynomial optimization problems are a topic of active research in the fields of algebraic geometry
and nonlinear optimization, with applications in dynamics and
control~\cite{ahmadi2011algebraic,majumdar2013control,mattingley2010real}, wireless
coverage~\cite{commander2008jamming,commander2007wireless}, and economics and game
theory~\cite{sturmfels2002solving}, to name a few.
These problems are, in general, nonlinear and nonconvex, making finding the global optimum a
difficult task.
Furthermore, polynomials with high degrees present a particular challenge since the nonlinearity is
more pronounced in such polynomials.
Existing optimization techniques often rely specifically on linear or quadratic structures in problems,
and even those methods that focus on general polynomial optimization problems often scale poorly
with the degree.

Global polynomial optimization methods can be roughly divided into two categories.
General-purpose approaches such as spatial branch-and-bound~\cite{horst2013global,gonzalez2022computational} are versatile and can call upon a variety of
sophisticated techniques in order to speed up the solving process.
However, these algorithms rely on linear or convex relaxations of the problem, and conventional
relaxations lack the means to efficiently capture polynomial nonlinearities.
Approaches based on nonnegativity certificates~\cite{shor1987class,nesterov2000squared,lasserre2001global,parrilo2003semidefinite} are better at leveraging the structure of an entire
polynomial, as opposed to its monomial terms, in order to obtain stronger dual bounds.
The search for the global optimum, though, requires constructing computationally expensive
hierarchies of relaxations.
The goal of this paper is to help bridge the existing gap between general-purpose algorithms and certificate-based relaxations
by (i) formulating a class of valid constraints to strengthen sums of nonnegative circuit
polynomials (SONC) relaxations in the presence of finite variable bounds, thus enabling such
relaxations to benefit from decreasing domain sizes in nodes of a branch-and-bound tree, and  (ii) developing a branch-and-bound algorithm that solves SONC relaxations next to linear
relaxations.

General-purpose solution methods such as spatial branch-and-bound can solve polynomial optimization
problems as long as some linear or convex relaxation is available, for example a
linear relaxation constructed by outer approximating the power and product terms appearing in a polynomial.
The Reformulation-Linearization Technique (RLT)~\cite{adams1986tight,adams1990linearization} can be
employed to strengthen such methods.
RLT produces a family of cutting planes in a lifted space, and it has been shown to yield strong
relaxations of polynomial problems~\cite{sherali1997new,sherali2012reduced,DalkiranSherali}.
The solver RAPOSA~\cite{gonzalez2022computational} implements an RLT-based branch-and-bound algorithm aimed at
efficiently solving polynomial optimization problems.

Another approach to polynomial optimization is based on results on nonnegativity of polynomials.
Such results usually rely on the cone of polynomials that are representable as sums of squares
(SOS), and originate in the works of Shor~\cite{shor1987class}, Nesterov~\cite{nesterov2000squared},
Lasserre~\cite{lasserre2001global} and Parrilo~\cite{parrilo2003semidefinite}.
A semidefinite program is solved in order to produce an SOS nonnegativity certificate for
$f(x) - \gamma$, seeking to maximize $\gamma$.
Further developments include improving the computational efficiency of SOS-based methods by
exploiting sparsity~\cite{waki2006sums,zheng2021sum}, employing restrictions of the SOS
cone~\cite{ahmadi2019dsos}, as well as
developing software for polynomial optimization such as GloptiPoly~\cite{henrion2009gloptipoly} and
SOSTOOLS~\cite{sostools}.

The SONC certificate, based on the decomposition of a polynomial into a sum of
nonnegative circuit polynomials, was proposed by Iliman and de Wolff~\cite{SONC}.
Unlike for SOS certificates, the computational cost of computing a SONC certificate does not depend
on the degree of the polynomial.
Furthermore, since the cones of SOS and SONC polynomials do not coincide or contain one
another~\cite{SONC}, for some problems a SONC decomposition yields better bounds than an SOS
decomposition.

The cone of sums of arithmetic geometric mean exponentials (SAGE)~\cite{chandrasekaran2016relative,chandrasekaran2017relative} provides an
alternative view of the SONC cone.
The membership of a polynomial in the SAGE cone can be decided by solving a convex relative entropy
program, and the method was successfully extended to constrained optimization
problems~\cite{murray2021signomial}.

These cones of nonnegative polynomials, however, in general only provide a dual bound.
While for SOS polynomials, there is a hierarchy of relaxations that converges to the global
optimum~\cite{lasserre2001global}, solving high levels of this hierarchy can have a prohibitive computational
cost.
An alternative approach to global optimization via nonnegativity certificates is to combine them
with a branch-and-bound algorithm.
The first such integrated algorithm was proposed by Seidler~\cite{seidler2021improved}.
This algorithm branches on signs of variables, so that additional terms can be identified as
positive.
However, convergence to the global optimum is not guaranteed.

In this work, we continue to investigate the potential for combining SONC-based relaxations with
branch-and-bound algorithms.
To this end, we first address the difficulties in utilizing variable bounds when applying SONC
relaxations.
We employ a Lagrangian relaxation approach to build the SONC relaxations of constrained optimization
problems and strengthen them by polynomial constraints derived from variable bounds.
We construct these constraints in a way that aims at achieving a
structure of the exponent set of the Lagrangian function that is well-suited for obtaining a SONC
certificate.
Adding these constraints enables the root node SONC relaxation to obtain finite dual bounds for
330 out of 349 instances from our test set, whereas the standard SONC relaxation found a finite dual
bound only for 9 instances.

The second contribution of this paper is an implementation of SONC relaxations within a
general-purpose spatial branch-and-bound algorithm.
We implement these relaxations in a relaxator plugin of the MINLP framework SCIP~\cite{SCIPoptsuite80}, which applies a preprocessing step,
adds polynomial-bound constraints and calls the polynomial optimization software POEM~\cite{poem:software} in order to
obtain SONC certificates.
Although our experiments showed that SONC relaxations are not yet competitive with state-of-the-art
linear relaxations, we observed a considerable improvement in the root node dual bound on 6
instances, with SONC relaxations closing up to $91$\% of the gap as compared to linear relaxations.

The rest of the paper has the following structure.
Sections~\ref{sec:sonc} and~\ref{sec:constrained} provide a summary of the theory of SONC
certificates and their use for polynomial optimization.
Section~\ref{sec:polynomial-bound} presents the main theoretical contribution of the paper: a new method for
incorporating variable bounds into a SONC relaxation.
In Section~\ref{sec:implementation}, we discuss the implementation of SONC relaxations within the
spatial branch-and-bound algorithm of SCIP.
Finally, in Section~\ref{sec:computational}, we report the results of our computational
experiments.


\section{SONC Certificates}\label{sec:sonc}

Consider a constrained polynomial optimization problem of the form
\begin{subequations}
\label{eqn:optproblem}
\begin{align}
\min_{\mathbf{x} \in \mathbb{R}^n}~~   &f(\mathbf{x}) = \sum_{\alpha \in \mathcal{A}(f)} f_{\alpha} \mathbf{x}^{\alpha} \label{eqn:optproblemObj} \\
\text{s.t.}~~ &g_i(\mathbf{x}) = \sum_{\alpha \in \mathcal{A}(g_i)} g_{i,\alpha} \mathbf{x}^{\alpha} \geq 0, \hspace{5mm} i=1,\ldots,m, \label{eqn:optproblemConst}
\end{align}
\end{subequations}
where $f_{\alpha}, g_{i,\alpha} \in \mathbb{R}$ are nonzero coefficients of the polynomials and
$\mathcal{A}(f), ~\mathcal{A}(g_i) \subset \mathbb{N}^n$, $i = 0,\dots,m$, are supports of polynomials $f$ and $g_i$.

In this formulation, monomials are written as $\mathbf{x}^{\alpha}:=
x_{1}^{\alpha_{1}}\cdot \ldots \cdot x_{n}^{\alpha_{n}}$.
An exponent $\alpha$ is a \emph{monomial square} if the term $f_{\alpha}\mathbf{x}^{\alpha}$
satisfies $f_{\alpha}\geq 0$ and $\alpha \in (2\mathbb{N})^n$.
The \emph{Newton polytope} New$(f)$ of a polynomial $f$ with support $\mathcal{A}(f)$ is defined to be
the convex hull of the exponents of $f$, that is $\text{New}(f) = \text{conv}(\mathcal{A}(f))$.
Let $V(f)$ denote the vertices of New$(f)$ and let $\Delta(f) = \mathcal{A}(f) \setminus \text{V}(f)$
denote the set of all non-vertex exponents.
We will refer to terms that correspond to exponents in $\Delta(f)$ as inner terms of $f$.
Further, we will denote the set of exponents that correspond to monomial square terms as MoSq$(f)$,
and the remaining set of exponents as $\overline{\text{MoSq}}(f) = \mathcal{A}(f) \setminus MoSq(f)$.

SONC certificates, similarly to SOS certificates, utilize a decomposition of a polynomial into
a sum of polynomials of a special structure, such that nonnegativity of such polynomials is easy to
prove.
In the case of SONC, these basic building blocks are circuit polynomials~\cite{SONC}:
\begin{definition}
A \textit{circuit polynomial} is a polynomial of the form
\begin{equation} \label{eq:circuitpoly}
f(\mathbf{x}) = \sum_{\alpha \in V(f)} f_{\alpha} \mathbf{x}^{\alpha} + f_{\beta} \mathbf{x}^{\beta},
\end{equation}
where the vertices $\alpha \in V(f)$ are affinely independent and are monomial squares, that is,
$V(f) \subseteq \text{MoSq}(f)$.
\end{definition}

The exponent $\beta$ of the inner term of a circuit polynomial can be uniquely written as a convex
combination of vertices:
\begin{equation} \label{eqn:barycoord}
\sum_{\alpha\in V(f)} \lambda_\alpha^{(\beta)} = 1 \quad \text{ and } \quad \sum_{\alpha\in V(f)} \lambda_\alpha^{(\beta)} \alpha = \beta.
\end{equation}
The weights $\lambda_\alpha^{(\beta)}$ are referred to as \textit{barycentric coordinates} of $\beta$. 
For a circuit polynomial, one can compute the circuit number
\begin{equation}\label{eqn:circuitnumb}
\theta_f\left(\beta\right) = \prod_{\alpha \in V(f)} \left(\dfrac{f_{\alpha}}{\lambda_\alpha^{(\beta)}} \right)^{\lambda_\alpha^{(\beta)}}.
\end{equation}

Nonnegativity of a circuit polynomial can be decided by comparing the circuit number to the coefficient of the inner term~\cite{SONC}:
\begin{theorem}\label{the:circuitnonneg}
A circuit polynomial $f$ as given in~\eqref{eq:circuitpoly} is nonnegative if and only if 
\begin{equation*}
\vert f_{\beta} \vert \leq \theta_f\left(\beta\right) \text{ and } \beta \notin \left(2 \mathbb{N} \right)^n \hspace{5mm} \text{ or } \hspace{5mm} f_{\beta} \geq -\theta_f\left(\beta\right) \text{ and } \beta \in \left(2 \mathbb{N}  \right)^n.
\end{equation*}
\end{theorem}

By utilizing circuit polynomials, Iliman and de Wolff~\cite{SONC} proposed a new class of
nonnegative polynomials:
\begin{definition}\label{def:sonc}
A polynomial $f$ is a SONC polynomial if it is of the form
\begin{equation}
f(\mathbf{x}) = \sum_{i=1}^\ell c_i f_i(\mathbf{x})
\end{equation}
where $c_i \geq 0$ are nonnegative coefficients and $f_i$ are nonnegative circuit polynomials for
all $i=1,\ldots,\ell$.
\end{definition}

Another class of polynomials that is of interest in the context of SONC decompositions is
the class of ST-polynomials.
Similarly to circuit polynomials, the vertices of ST-polynomials are monomial squares that are
affinely independent, or, in other words, form a simplex.
The difference is that ST-polynomials can have multiple inner terms.

\begin{definition} \label{def:STpoly}
A polynomial $f$ is an \emph{ST-polynomial}~\cite{Lower} if it has the form
\begin{equation} \label{eqn:STpoly}
f(\mathbf{x}) = \sum_{\alpha \in V(f)} f_{\alpha} \mathbf{x}^{\alpha} + \sum_{\beta \in \Delta(f)} f_{\beta} \mathbf{x}^{\beta}
\end{equation}
such that New$(f)$ is a simplex whose vertices $\alpha \in V(f)$ are monomial squares, that is,
$V(f) \subseteq \text{MoSq}(f)$.
\end{definition}

For all $\beta \in \Delta(f)$ there exist $\lambda_\alpha^{(\beta)} \geq 0$, $\alpha
\in V(f)$, forming the unique convex combination~\eqref{eqn:barycoord}.
The vertex set $V(f)$ of the simplex is referred to as a \emph{cover} of the inner term $\beta$.
We will say that $\beta$ is \emph{covered} by $V(f)$.

\begin{figure}
\centering
\includegraphics{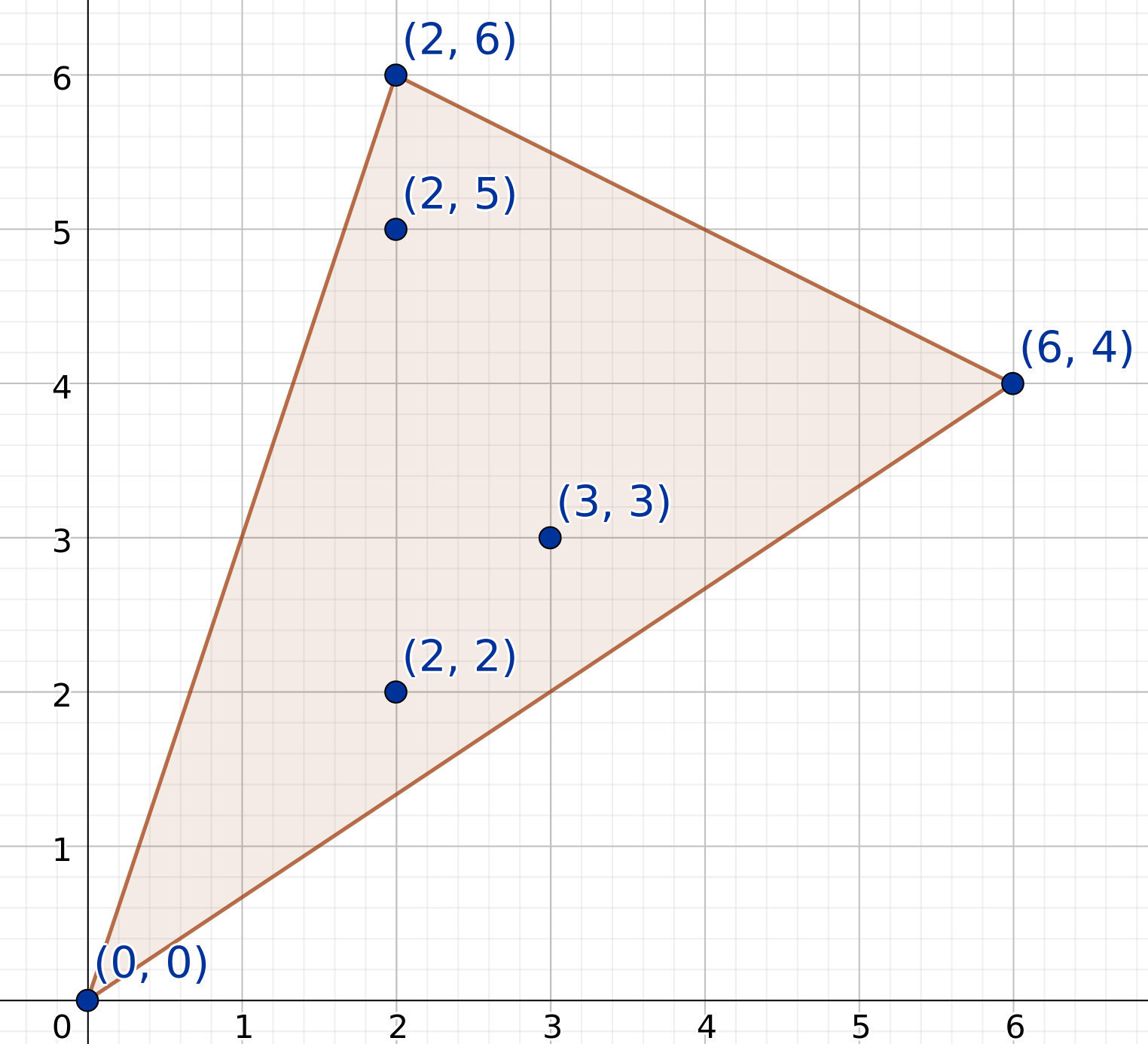}
\smallskip
\caption{Exponents and the Newton polytope of the polynomial $f(x_1,x_2) = x_1^6x_2^4 + x_1^3x_2^3 + x_1^2x_2^6 + x_1^2x_2^5 + x_1^2x_2^2 + 1$.}
\label{fig:exponents}
\end{figure}

Figure~\ref{fig:exponents} shows points corresponding to the exponents of the polynomial
$f(x_1,x_2) = x_1^6x_2^4 + x_1^3x_2^3 + x_1^2x_2^6 + x_1^2x_2^5 + x_1^2x_2^2 + 1$.
The coloured region represents the Newton polytope.
Since the vertices $(6,4)$, $(2,6)$ and $(0,0)$ are even, correspond to terms with positive
coefficients and are affinely independent, $f(x_1,x_2)$ is an ST-polynomial.
Exponents of the inner terms $(2,5)$, $(3,3)$ and $(2,2)$ can be expressed as unique convex
combinations of the vertices.

It is possible to split an ST-polynomial into circuit polynomials by taking the same Newton
polytope for each inner term and splitting the coefficients of the monomial squares among the
circuit polynomials.
The existence of such a decomposition is a proof of nonnegativity.

\begin{theorem}\label{thm:stnonneg} (\cite[Theorem 3.1]{Lower})
An ST-polynomial $f(\mathbf x)$ is a SONC polynomial if for every $(\beta, \alpha) \in \Delta(f)
\times V(f)$ there exist $a_{\beta,\alpha} \geq 0$ such that
\begin{gather*}
\rvert f_\beta\rvert \leq \prod_{\alpha \in nz(\beta)}\left(\frac{a_{\beta,\alpha}}{\lambda_\alpha^{(\beta)}}\right)^{\lambda_\alpha^{\beta}},\\
f_{\alpha} \geq \sum_{\beta \in \Delta(f)} a_{\beta,\alpha}.
\end{gather*}
These $a_{\beta,\alpha}$ are the weights in the following SONC decomposition:
$$f(\mathbf x) = \sum_{\beta \in \Delta(f)} \left(\sum_{\alpha \in nz(\beta)} a_{\beta,\alpha}\mathbf x^{\alpha} + f_\beta \mathbf x^\beta\right),$$
where $nz(\beta)$ denotes all exponents $\alpha\in V(f)$ that correspond to nonzero barycentric coordinates $\lambda_\alpha^{(\beta)}$.
\end{theorem}


\section{Polynomial Optimization via SONC}\label{sec:constrained}

In this section we restate known results on SONC polynomials and optimization methods based on SONC
relaxations.
Polynomial optimization problems can be stated as nonnegativity problems, since one can write the
problem of minimizing $f(\mathbf x)$ equivalently as
\begin{equation}\label{eq:optnonneg}\sup \{ \gamma \in \mathbb{R}: \hspace{2mm} f(\mathbf x)-\gamma\geq 0\}.\end{equation}
This problem, however, is as hard as the original problem.
Requiring instead that some certificate of nonnegativity exists for $f(\mathbf x) - \gamma$
results in a relaxation of the original problem.
If a SONC decomposition is used as a nonnegativity certificate, then the relaxation is
\begin{equation}\label{eq:optnonnegsonc}
\sup \{\gamma \in \mathbb{R}: f(\mathbf{x}) - \gamma \text{ is SONC}\}.
\end{equation}


\subsection{Lower Bounds for Unconstrained Optimization over an ST-Polynomial}\label{sec:GPunconst}

This section shows how to write the problem of finding the optimal SONC decomposition of an
unconstrained polynomial optimization problem as a convex optimization problem.
We begin by recalling the notion of geometric programs (GPs).

\begin{definition}~\cite{BoydVan}
A \textit{monomial} is defined as a function $q:\mathbb{R}^{n}\rightarrow \mathbb{R}$ of the form \mbox{$q(\mathbf{x})= c x_1^{a_1} \cdot \ldots \cdot x_n^{a_n}$} with $c>0$.
A \textit{posynomial} is defined as a sum of monomials.
A \textit{geometric program (GP)} is an optimization problem of the form
\begin{mini}
{}{p_0(\mathbf{x})}
{\label{eqn:generalgp}}{}
\addConstraint{p_i(\mathbf{x})}{\leq 1, \quad }{\text{ for }i=1,\ldots,r}
\addConstraint{q_j(\mathbf{x})}{= 1, \quad }{\text{ for }j=1,\ldots,s},
\end{mini}
where $p_0,\ldots,p_r$ are posynomials and $q_1,\ldots,q_s$ are monomials.
\end{definition}
Applying a logarithmic transformation to a GP results in a convex optimization
problem~\cite{boyd2007tutorial} which is equivalent to the GP; in particular, applying a reverse
transformation to an optimal solution of the transformed problem gives an optimal solution of the
original GP.

Let $f$ be an ST-polynomial and assume that $\Delta(f) \cap \text{MoSq}(f) = \emptyset$.
This can be achieved by disregarding all monomial squares that are not vertices of New$(f)$.
This change preserves the validity of a lower bound on the optimal solution $f^*$ since monomial
squares are always nonnegative~\cite{GP}.
The following theorem formulates a GP based on the nonnegativity conditions from
Theorem~\ref{thm:stnonneg}.

\begin{theorem}\cite[Corollary 2.7]{GP}\label{Cor:GPunconstrained}
Let $f$ be an ST-polynomial and introduce variables $a_{\beta,\alpha} > 0$ for every $\beta \in
\Delta(f)$ and $\alpha \in V(f)$.
Then the SONC lower bound on $f$ is given by $f_{\text{SONC}} = f_0 - \gamma$,
where $\gamma$ is the solution of the following GP:
\begin{mini!}
{}{\sum_{\substack{\beta \in \Delta(f)\\ \lambda_{0}^{(\beta)} \neq 0}} \lambda_{0}^{(\beta)} \cdot \vert f_{\beta} \vert ^{\dfrac{1}{\lambda_{0}^{(\beta)}}} \cdot \prod_{\substack{j \in nz(\beta) \\ \alpha\neq 0}} \left(\dfrac{\lambda_{\alpha}^{(\beta)}}{a_{\beta,\alpha}} \right)^{\dfrac{\lambda_{\alpha}^{(\beta)}}{\lambda_{0}^{(\beta)}}}}
{\label{eqn:GPunconstrained}}{}
\addConstraint{\sum_{\beta \in \Delta(f)} \dfrac{a_{\beta,\alpha}}{f_{\alpha}}}{\leq 1 \text{ for all } \alpha \in V(f) \setminus 0}
\addConstraint{\vert f_{\beta} \vert \prod_{\alpha \in nz(\beta)} \left( \dfrac{\lambda_{\alpha}^{(\beta)}}{a_{\beta,\alpha}} \right)^{\lambda_{\alpha}^{(\beta)}}}{\leq 1 \text{ for all } \beta \in \Delta(f) \text{ with } \lambda_{0}^{(\beta)}=0}.
\end{mini!}
\end{theorem}

\subsection{Lower Bounds for Constrained Optimization with an ST-Polynomial Lagrangian Function}\label{sec:GP}

Dressler et al.~\cite{GP} extend the method described above to the constrained case by utilizing
Lagrangian relaxations.
Consider an optimization problem of the form \eqref{eqn:optproblem}.
Then the Lagrangian function of the problem has the form
\begin{equation}
L\left(\mathbf{x},\mathbf{\mu}\right) = f\left(\mathbf{x}\right) - \sum_{i=1}^{m} \mu_{i} g_{i} \left(\mathbf{x}\right),
\end{equation}
where $\mu_{i} \geq 0$, $i=1,\ldots,m$ are Lagrangian multipliers.
Let $\mu_0:=1$ and $g_0:=-f$.
The Lagrangian then becomes
\begin{equation}\label{eqn:lagrangianUsed}
L(\textbf{x},\mathbf{\mu}) = - \sum_{i=0}^{m} \mu_{i} g_{i}(\mathbf{x}).
\end{equation}
%
%

The coefficients of the polynomial~\eqref{eqn:lagrangianUsed} depend on $\mu$, and term
cancellation may lead to some of the monomials vanishing for certain values of $\mu$.
Therefore, the definitions of exponent sets need to account for this.
Thus, the support $\mathcal{A}(L)$ is defined as the union of supports of individual polynomials
defining the objective and the constraints: $\mathcal{A}(L) = \cup_{i=0}^m \mathcal{A}(g_i)$.
The definitions of the vertices of the Newton polytope and the exponents corresponding to inner
terms are analogous: $V(L) = \cup_{i=0}^m V(g_i)$ and $\Delta(L) = \cup_{i=0}^m \Delta(g_i)$.

We assume that $\Delta(L)$ does not contain exponents corresponding to monomial square terms.
If this is not the case, then, since monomial square terms are always nonnegative, we can disregard
monomial square inner terms and still obtain a valid lower bound.

Further, we assume that $L$ is an ST-polynomial, and write it in the form
\begin{equation}\label{eqn:lagrangianST}
L(\mathbf{x},\mathbf{\mu}) = \sum_{\alpha \in V(L)} L(\mathbf{\mu})_{\alpha} \mathbf{x}^{\alpha} + \sum_{\beta \in \Delta(L)} L(\mathbf{\mu})_{\beta}\mathbf{x}^{\beta}.
\end{equation}

%
For a fixed $\mu$, the problem of optimizing the Lagrangian reduces to a similar GP that is solved
in the unconstrained case.
Let $\gamma(\mu)$ denote the optimal value of problem~\eqref{eqn:GPunconstrained} for a given
$\mu$.
The goal is to find the best possible lower bound $\gamma^*$ over all nonnegative $\mu$,
that is, \mbox{$\gamma^* = \sup \{ \gamma(\mu) : \mu \in \mathbb{R}^n_+\}$}.
However, directly writing this in the form~\eqref{eqn:GPunconstrained}
does not produce a GP, and additional relaxation steps are necessary in order to obtain a
GP~\cite{GP}.

As before, let $\lambda_\alpha^{(\beta)}$ denote the barycentric coordinates for each $\beta \in
\Delta(L)$ with respect to the vertices $\alpha \in V(L)$.
For each $\beta \in \Delta(L)$, we introduce new variables $a_{\beta,\alpha} > 0$ for each $\alpha \in V(L)$,
and $b_{\beta} \geq 0$.
For each constraint $g_{i}(\mathbf{x})$, \mbox{$i = 1,\ldots,m$}, we denote by $g_{i,\alpha}$
the coefficient corresponding to the monomial $\mathbf{x}^{\alpha}$.
Then optimizing the Lagrangian over $\mu$ and $\mathbf{x}$ is equivalent to solving the following
optimization problem with variables $\mathbf{\mu}$, $a_{\beta}$ and $b_{\beta}$:
\begin{mini!}
{\substack{\mathbf{\mu}\geq 0, \\a_{\beta} > 0, b_{\beta} \geq 0}}{\sum_{i=1}^{m} \mu_{i} g_{i,0} + \sum_{\substack{\beta \in \Delta(L)\\ \lambda_{0}^{(\beta)} \neq 0}} \lambda_{0}^{(\beta)} \cdot b_{\beta}^{\dfrac{1}{\lambda_{0}^{(\beta)}}} \cdot \prod_{\substack{\alpha \in nz(\beta) \\ \alpha\neq 0}} \left(\dfrac{\lambda_{\alpha}^{(\beta)}}{a_{\beta,\alpha}} \right)^{\dfrac{\lambda_{\alpha}^{(\beta)}}{\lambda_{0}^{(\beta)}}}\label{eqn:objectiveGPone}}
{\label{eqn:GPproblem}}{}
\addConstraint{\sum_{\beta \in \Delta(L)} a_{\beta,\alpha}}{\leq L(\mathbf{\mu})_{\alpha}\text{ for all } \alpha\neq 0}
\addConstraint{\prod_{\alpha \in nz(\beta)} \left( \dfrac{a_{\beta,\alpha}}{\lambda_{\alpha}^{(\beta)}} \right)^{\lambda_{\alpha}^{(\beta)}}}{\geq b_{\beta}\text{ for all } \beta \in \Delta(L) \text{ with } \lambda_{0}^{(\beta)}=0}
\addConstraint{\vert L(\mathbf{\mu})_{\beta} \vert }{\leq b_{\beta}\text{ for all } \beta \in \Delta(L) \text{ with } \lambda_{0}^{(\beta)}\neq 0. \label{eqn:constraintGPthree}}
\end{mini!}

As before, $nz(\beta)$ denotes all exponents $\alpha \in V(L)$ with non-zero barycentric
coordinates in the convex combination that forms $\beta$, i.e. $nz(\beta) = \{\alpha\in V(L) ~\rvert~
\lambda_\alpha^{(\beta)}\neq 0\}$.
This optimization problem was stated in \cite{Lower} and the following statements were shown,
see~\cite[Theorem 5.1 and 5.2]{Lower}.

\begin{theorem}\label{the:GPone}
Let $L$ be an ST-polynomial of the form \eqref{eqn:lagrangianST}.
Assume that every coefficient $L(\mathbf{\mu})_{\alpha}$ consists of only one summand for each
$\alpha \in V(L)$ and is strictly positive.
Moreover, assume that for each $\beta \in \Delta(L)$ the coefficient $L(\mathbf{\mu})_{\beta}$ has
only positive terms and $g_{i,0} \geq 0$ for all $i=1,\ldots,m$.
Then the problem~\eqref{eqn:GPproblem} is a GP for $\mathbf{\mu}>0$.

The optimization problem~\eqref{eqn:GPproblem} yields a lower bound for
problem~\eqref{eqn:optproblem}.
More precisely, if $\gamma^*$ is the solution of the GP~\eqref{eqn:GPproblem} and $f^{*}$ is the
solution of the original problem~\eqref{eqn:optproblem}, then we have 
$$f_{0} - \gamma^* \leq f^{*}.$$
\end{theorem}


However, Theorem~\ref{the:GPone} places considerable restrictions on the coefficients.
Hence, a relaxation of this problem was introduced in~\cite{GP}.
The key step is to redefine the objective~\eqref{eqn:objectiveGPone} and the
constraint~\eqref{eqn:constraintGPthree} in order to obtain a GP under weaker assumptions.

First, we split up the coefficients $L(\mathbf{\mu})_{\beta}$ into their positive and their negative
parts, defining $L(\mathbf{\mu})_{\beta} = L(\mathbf{\mu})_{\beta}^{+} - L(\mathbf{\mu})_{\beta}^{-}$
with
\begin{equation*}
L(\mathbf{\mu})_{\beta}^{+} = \sum_{\substack{i=1,\ldots,m \\ g_{i,\beta}>0}} \mu_{i} g_{i,\beta} \hspace{5mm}\text{ and } \hspace{5mm}
L(\mathbf{\mu})_{\beta}^{-} = - \sum_{\substack{i=1,\ldots,m \\ g_{i,\beta}<0}} \mu_{i} g_{i,\beta},
\end{equation*}
where $g_{i,\beta}$ denotes the coefficient corresponding to $\mathbf{x}^\beta$ of $g_i$.
Using this decomposition, constraint~\eqref{eqn:constraintGPthree} can be changed to
\mbox{$\max \{ L(\mathbf{\mu})_{\beta}^{+}, L(\mathbf{\mu})_{\beta}^{-}\}\leq b_{\beta}$}.
Further, in the objective function, instead of all $g_{i,0}$ for $i=1,\ldots,m$, we only consider
those coefficients of the monomial $\mathbf{x}^{0}$ that are positive in the corresponding $g_i$
and thus negative in $L(\mathbf x, \mu)$, resulting in $g_{i,0}^{+}=\max \{g_{i,0},0\}$.
This produces the following optimization problem:
\begin{mini!}
{\substack{\mathbf{\mu}\geq 0, \\a_{\beta} > 0, b_{\beta} \geq 0}}{\sum_{i=1}^{m} \mu_{i} g_{i,0}^{+} + \sum_{\substack{\beta \in \Delta(L)\\ \lambda_{0}^{(\beta)} \neq 0}} \lambda_{0}^{(\beta)} \cdot b_{\beta}^{\dfrac{1}{\lambda_{0}^{(\beta)}}} \cdot \prod_{\substack{\alpha \in nz(\beta) \\ \alpha\neq 0}} \left(\dfrac{\lambda_{\alpha}^{(\beta)}}{a_{\beta,\alpha}} \right)^{\dfrac{\lambda_{\alpha}^{(\beta)}}{\lambda_{0}^{(\beta)}}}}
{\label{eqn:GPproblemtwo}}{}
\addConstraint{\sum_{\beta \in \Delta(L)} a_{\beta,\alpha}}{\leq L(\mathbf{\mu})_{\alpha}\text{ for all } \alpha\neq 0}
\addConstraint{\prod_{\alpha \in nz(\beta)} \left( \dfrac{a_{\beta,\alpha}}{\lambda_{\alpha}^{(\beta)}} \right)^{\lambda_{\alpha}^{(\beta)}}}{\geq b_{\beta}\text{ for all } \beta \in \Delta(L) \text{ with } \lambda_{0}^{(\beta)}=0}
\addConstraint{L(\mathbf{\mu})_{\beta}^{+}}{\leq b_{\beta} \text{ for all } \beta \in \Delta(L)}
\addConstraint{L(\mathbf{\mu})_{\beta}^{-}}{&\leq b_{\beta} \text{ for all } \beta \in \Delta(L)}.
\end{mini!}

This optimization problem is a relaxation of \eqref{eqn:GPproblem}.
The following theorem~\cite{GP} shows that it is a GP under weaker assumptions.

\begin{theorem}\label{the:GP}
Let $L$ be an ST-polynomial of the form~\eqref{eqn:lagrangianST}.
Assume that for every $\alpha \in V(L) \setminus 0$ the coefficient $L(\mathbf{\mu})_{\alpha}$
has exactly one strictly positive term.
Then the optimization problem~\eqref{eqn:GPproblemtwo} is a geometric program for $\mathbf{\mu}>0$.
Moreover, the GP provides a lower bound on the solution of the original optimization
problem~\eqref{eqn:optproblem}.
Denoting the solution of~\eqref{eqn:GPproblemtwo} by $\gamma_{\text{SONC}}$ and the solution
of~\eqref{eqn:GPproblem} by $\gamma^*$, we get
\begin{equation}
f_0 - \gamma_{\text{SONC}} \leq f_0 - \gamma^* \leq f^*,
\end{equation}
where $f^*$ is the optimal solution of \eqref{eqn:optproblem}.
\end{theorem}

\subsection{Lower Bounds for Arbitrary Constrained Polynomial Optimization Problems}

The relaxation methods described so far in this section are designed for ST-polynomials.
Recall that, by Definition~\ref{def:STpoly}, an ST-polynomial satisfies two conditions: the
vertices of the Newton polytope must correspond to monomial square terms and form a simplex.
The first condition is a necessary condition for nonnegativity of a polynomial~\cite{reznick1978extremal}.
The second condition, however, may not hold for some nonnegative polynomials, and one can obtain a
SONC certificate by applying an extended technique based on SONC certificates.

Consider a polynomial $f$ that satisfies the first condition and violates the second condition.
The idea of the extended method~\cite{seidler2018experimental,GP}, is to write $f$ as a sum of
ST-polynomials.
To this end, the algorithm presented in~\cite{seidler2018experimental} covers all terms of $f$
that are not monomial squares by monomial square terms by solving a linear program for each
non-monomial-square term.
Then it splits the coefficients of the terms of $f$ between the new polynomials so that summing
these polynomials up yields $f$.
The coefficients can either be split evenly, though this does not guarantee an optimal splitting,
or one can solve a modified version of the SONC optimization problem in order to find an optimal
coefficient distribution.


\section{Polynomial-Bound Constraints}\label{sec:polynomial-bound}

As mentioned in the previous section, a necessary condition for the nonnegativity of a polynomial is that all vertices $V(f)$ are
monomial squares, that is, $V(f) \subseteq \text{MoSq}(f)$.
Therefore, when working with an arbitrary polynomial, it may be necessary to reformulate the lower
bounding problem so that each term that is not a monomial square is covered by even exponents.
Moreover, a relaxation that is solved as part of a branch-and-bound algorithm must be able to
utilize tighter variable bounds in the nodes of the tree in order to obtain lower bounds of improving quality.

We assume that each variable has finite bounds.
However, directly adding a term corresponding to a variable bound, that is, $(x_i-l_i)$ or
$(u_i-x_i)$, to the Lagrangian function does not lead to an improvement in the SONC bound since $x_i$ is not
a monomial square, and is therefore treated by the SONC relaxation as a `negative' term.

In order to circumvent this issue, we derive polynomial constraints on $x_i$ that are valid with
respect to the bounds on $x_i$.
The goal is to find a polynomial $h(x_i)$ that is nonnegative on $[l_i,u_i]$ such
that, when added to the Lagrangian, $h(x_i)$ adds monomial squares to it, but not new
non-monomial-square terms.

Let $L'$ denote the Lagrangian function after the replacement of variable bounds with
polynomial-bound constraints.
The choice of such a polynomial is not unique and must take the current structure of the exponent
set of $L$ into account.

Let $A'$ be a matrix defining exponents introduced in polynomial-bound constraints, where an entry
$\alpha'_{ji}$ denotes the exponent for variable $j$ in the $i$th polynomial-bound constraint, and a column
$\alpha'_{*i}$ defines the exponent vector corresponding to the $i$th polynomial-bound constraint.
We propose to use polynomial constraints of the form
\begin{equation*} \label{eqn:boundcons}
\mathbf{x}^{\alpha'_{*i}} \leq \max \{\vert l_i \vert, \vert u_i \vert \}^{\alpha'_{ii}},
\end{equation*}
where $\alpha'_{ji} = 0$ if $j \neq i$, and refer to
them as \emph{polynomial-bound constraints}.

The choice of exponents $A'$ will aim at ensuring that all vertices of the Newton polytope
New($L'$) are monomial squares.
In other words, this requires a set of exponents such that these exponents together
with the even exponents $\alpha \in V(L)$ provide a valid cover for each $\beta\in\overline{\text{MoSq}}(L)$.

Let $\mathcal{A}' = \{\alpha_{*1}, \dots, \alpha_{*n}\} \subset (2\mathbb{N})^n$ denote the set of new exponents, and let $V(L')$
denote the set of all vertices of the Newton polytope of $L$ after the transformation of variable
bounds into polynomial-bound constraints.
We require that the vertices $V(L')$ provide a cover for all $\beta \in \overline{\text{MoSq}}(L)$, that is,

\begin{align}\label{eqn:OPexponent}
&\forall~\beta \in \overline{\text{MoSq}}(L) ~\exists~ \lambda^{(\beta)} \in \mathbb{R}_+^{n+1}:
\beta= \sum_{\alpha \in V(L')} \lambda_\alpha^{(\beta)} \alpha \text{ and } \sum_{\alpha \in V(L')} \lambda_\alpha^{(\beta)}= 1.
\end{align}

The following theorem provides a choice of the set $\mathcal{A}'$ that ensures that the above
requirement holds:

\begin{theorem}\label{the:OPexponent}
Let $A' \in \mathbb{N}^{n \times n}$ be a matrix defining the exponents in polynomial-bound
constraints, and let $\mathcal{A}' = \{\alpha_{*1}, \dots, \alpha_{*n}\}$, where $\alpha'_{ji} = 0$
if $j \neq i$, and
$$\alpha'_{ii} \geq (n+(n \mod 2))\cdot \max_{\beta \in \overline{\text{MoSq}}(L)} ||\beta||_{\infty}.$$
Then $\mathcal{A}' \subset (2\mathbb{N})^n$ and~\eqref{eqn:OPexponent} holds, that is, all terms
$\beta\in \overline{\text{MoSq}}(L)$ are covered by some subset of the exponents $V(L')$.
\end{theorem}
\begin{proof}

The condition $\mathcal{A}' \in (2\mathbb{N})^n$ holds since $(n + (n \mod 2))$ is an even number.
We now need to prove the existence of barycentric coordinates $\lambda^{(\beta)}$
satisfying~\eqref{eqn:OPexponent}.

For each $\beta \in \overline{\text{MoSq}}(L)$, consider the following barycentric coordinates:
\begin{align*}
\lambda_{\alpha'_{i*}}^{(\beta)} &= \dfrac{\beta_i}{\alpha'_{ii}}  & \text{ for } i=1,\ldots,n,\\
\lambda_0^{(\beta)} &= 1-\sum_{\alpha' \in \mathcal{A}'} \lambda_{\alpha'}^{(\beta)}, &\\
\lambda_\alpha^{(\beta)} &= 0 &\text{ for }\alpha \in V(L). 
\end{align*}
Then $\beta$ can be written as
\begin{align*}
\beta &= \left(\begin{array}{c} 
\lambda_{\alpha'_{1*}}^{(\beta)} \alpha'_{11} \\ \vdots \\ \lambda_{\alpha'_{n*}}^{(\beta)} \alpha'_{nn}
\end{array}\right)
 = \lambda_0^{(\beta)} \mathbf{0} + \sum_{\alpha \in V(L)} 0\cdot \alpha +\sum_{\alpha' \in \mathcal{A}'} \lambda_{\alpha'}^{(\beta)} \alpha' = \\
 & = \lambda_0^{(\beta)} \mathbf{0} + \sum_{\alpha \in V(L)} \lambda_\alpha^{(\beta)}\alpha +\sum_{\alpha' \in \mathcal{A}'} \lambda_{\alpha'}^{(\beta)} \alpha'.
\end{align*}
Thus, the condition $\beta= \sum_{\alpha \in V(L')} \lambda_\alpha^{(\beta)} \alpha$
of~\eqref{eqn:OPexponent} is fulfilled.
It remains to show that $\lambda_{\alpha}^{(\beta)} \geq 0$ for all $\alpha \in V(L')$, and $\sum_{\alpha \in V(L')} \lambda_\alpha^{(\beta)} = 1$.

The condition $\sum_{\alpha \in V(L')} \lambda_{\alpha}^{(\beta)} = 1$ and the
nonnegativity of all components of $\lambda$ except for $\lambda_0^{(\beta)}$ follow
directly from the definition of $\lambda$.
To prove that $\lambda_0^{(\beta)} \geq 0$, observe that
$\alpha'_{ii} \geq (n+(n \mod 2))\cdot \max_{\beta \in \overline{\text{MoSq}}(L)} ||\beta||_{\infty}$ and
$\lambda_{\alpha'_{i*}}^{(\beta)} = \dfrac{\beta_i}{\alpha'_{ii}}$, and therefore
$$\lambda_{\alpha'_{i*}}^{(\beta)} \leq \dfrac{\beta_i}{(n+(n \mod 2))\cdot \max_{\beta \in \overline{\text{MoSq}}(L)} ||\beta||_{\infty}} \leq \dfrac{1}{n+(n \mod 2)}.$$
This implies that
\begin{align*}
\sum_{\alpha' \in \mathcal{A}'} \lambda_{\alpha'}^{(\beta)} &\leq \dfrac{n}{n+(n \mod 2)} \leq 1 \\
\Rightarrow \lambda_0^{(\beta)} &= 1 - \sum_{\alpha' \in \mathcal{A}'} \lambda_{\alpha'}^{(\beta)} \geq 0.
\end{align*}
\end{proof}

The above theorem provides a sufficient condition on the exponents $\alpha'$ in order to guarantee
a cover for all inner terms.
Since the cost of computing a SONC certificate does not depend on the degree of a polynomial,
introducing terms with high degrees is not an issue.
The practical choice of exponent values will be further discussed in Section~\ref{sec:computational}.

\section{SONC Relaxations in a Branch-and-Bound Algorithm}\label{sec:implementation}

We have implemented an experimental algorithm that combines an LP-based branch-and-bound algorithm
with SONC relaxations.
The goal was to assess the potential of integrating the branch-and-bound and SONC-based approaches
to constrained polynomial optimization.
Our algorithm solves SONC relaxations in some nodes of the branch-and-bound tree and, if they yield
a better dual bound than the standard LP relaxations, uses this bound.

\subsection{Software: SCIP and POEM}

We used the polynomial optimization software POEM in order to solve SONC
relaxations.
POEM mainly employs SONC nonnegativity certificates, although it also supports SAGE and SOS
certificates for unconstrained problems.
To solve SONC relaxations, it constructs a geometric problem and applies a logarithmic
transformation in order to obtain a convex problem, which it then passes to the solver ECOS~\cite{domahidi2013ecos}
via CVXPY~\cite{diamond2016cvxpy}.
The main focus of POEM is currently on unconstrained optimization, but it also provides
functionality for constrained problems.

The general-purpose solver SCIP~\cite{SCIPoptsuite80} provided the branch-and-bound algorithm and
called POEM in order to solve SONC relaxations.
SCIP is a Constraint Integer Programming (CIP) solver, which means that its design allows it to
handle any problems where fixing all integer variables results in a linear or nonlinear program.
In particular, SCIP can solve problems belonging to two important subclasses of the CIP problem
class: mixed-integer linear programs (MILPs) and mixed-integer nonlinear programs (MINLPs).
SCIP implements a spatial branch-and-bound algorithm which, by default, solves linear programming
(LP) relaxations at each node of the branch-and-bound tree.

SCIP has a plugin-based structure, wherein plugins implementing various components of the solving
process such as presolving techniques, primal heuristics, domain propagation techniques, cutting
plane separation, and others techniques, are coordinated by the core of the solver.
Users can add new plugins.

For the purposes of our implementation, we created a new relaxator plugin.
Relaxator plugins enable the solving of custom relaxations instead of default LP relaxations.
When called, a relaxator plugin can return a dual bound or report infeasibility of the current
node.
Additionally, relaxators may provide primal solution candidates, reduce variable domains and add
branching candidates and cutting planes.

Further, since polynomial problems are nonlinear problems, the handling of nonlinear constraints in
SCIP is relevant to our implementation.
Nonlinear expressions in SCIP are represented by expression graphs, where the nodes represent
operations and the arcs represent the flows of computation.
SCIP constructs an extended formulation of the problem in a lifted space in order to construct LP
relaxations.
It stores this extended formulation alongside the original formulation of the problem.

\subsection{Algorithm and Implementation}

We implement SONC relaxations as a relaxator plugin in SCIP via the Python interface PySCIPOpt\footnote{Development branch SONCRelax at \url{https://github.com/scipopt/PySCIPOpt}}~\cite{maher2016pyscipopt}.
The main steps that the relaxator performs are the following:
\begin{itemize}
\item it inspects the structure of the problem in order to decide if a SONC relaxation should be
applied to it;
\item reverts the creation of an objective variable to represent the nonlinear part of the objective
function;
\item converts the problem into a matrix representation, where each polynomial is given by a vector
of term coefficients and a matrix comprised of exponent vectors for each monomial term;
\item if needed, adds polynomial-bound constraints;
\item passes the problem to POEM, which solves the SONC relaxation and returns either a dual bound
or a failure status;
\item and reports the results back to SCIP.
\end{itemize}
The relaxator is called in the root node and in every 10 nodes of the branch-and-bound tree.
In the following, we will further explain some steps of the algorithm.

\paragraph{Problem inspection}
First, the relaxator analyses the expressions in nonlinear constraints in order to determine
whether they are polynomials in the original formulation of the problem.
The SONC relaxation is only solved when at least half of constraints are polynomial constraints.
Linear constraints are not counted as polynomial constraints, but are added to the relaxation in
the same way as polynomial constraints are.
The relaxator does not use constraints that cannot be represented as polynomial constraints.

\paragraph{Objective variable removal}
SCIP handles nonlinear objective functions by reformulating the problem so that to move the
nonlinear part of the objective function into a constraint.
That is, SCIP applies the following reformulation:

\begin{minipage}{.45\linewidth}
\begin{align*}
&\min f(\mathbf x) + \mathbf c^T\mathbf x\\
&\text{subject to } &\rightarrow\\
&g(\mathbf x) \geq 0,
\end{align*}
\end{minipage}
\begin{minipage}{.45\linewidth}
\begin{align*}
&\min y + \mathbf c^T\mathbf x\\
&\text{subject to }\\
& y - f(\mathbf x) \geq 0, ~~g(\mathbf x) \geq 0.
\end{align*}
\end{minipage}
\bigskip

\noindent
This, however, introduces a non-square monomial ($y$) which is not covered by any monomial squares
since $y$ does not participate in any other monomial terms.
Therefore, the relaxator reverses this reformulation by detecting variables that are present only
in the objective and one constraint (which imposes a restriction on how much the variable can
decrease), removing such variables and moving the rest of the corresponding constraints to the
objective function.

\paragraph{Polynomial-bound constraints}
The relaxator calls POEM to compute the vertices of the Newton polytope of the Lagrangian function.
If the vertices are monomial squares and form a simplex, that is, the polynomial is an
ST-polynomial, then the algorithm will preserve the existing structure of the polynomial.
The relaxator may still add polynomial-bound constraints $(\max\{\rvert\ell\rvert, \rvert u\rvert\})^{\alpha'} - \mathbf x^{\alpha'}
\geq 0$, but it uses only those exponents $\alpha'$ that belong to the original exponent set $\mathcal{A}(L)$.
Thus, no new monomial terms are introduced, and the Lagrangian function remains an ST-polynomial.

If the Lagrangian function is not an ST-polynomial, then either some of the vertices are
non-monomial-squares or the vertices are affinely dependent, or both.
If some vertices are non-monomial-squares, then the necessary condition for the nonnegativity of
the polynomial is violated, and the addition of polynomial-bound constraints is necessary in order
to obtain a finite dual bound by applying a nonnegativity certificate.
Adding such constraints, however, does not guarantee that the vertices of the new polynomial form a
simplex.

POEM in its current implementation is unable to solve relaxations of constrained
optimization problems where the vertices of the Lagrangian function polynomial are affinely
dependent.
Recall that in the unconstrained case, a non-ST-polynomial is split up into several ST-polynomials,
and then dual bounds are computed for each of those polynomials and combined to obtain the dual
bound on the entire polynomial.
In the constrained case, however, the polynomial to be optimized has variable coefficients that
depend on the Lagrangian multipliers $\mu$.
If the same variable $\mu_j$ is involved in several ST-polynomials obtained by such splitting,
then, in general, it will have different values in solutions of the corresponding problems.
In order to obtain a valid bound, however, the value of $\mu_j$ must be similar for all
ST-polynomials.
Ensuring this is by itself a challenging task, and is not done in the current implementation.

Therefore, given a polynomial with affinely dependent monomial square vertices of the Newton
polytope, we always use the terms obtained from polynomial-bound constraints as the cover for all
inner points.
These vertices always form a simplex that covers all inner points.
Other vertices are disregarded so that to preserve this property.
This is a valid relaxation, since monomial square vertices can only increase the dual bound.
However, bound quality is worsened when such vertices are discarded.

Thus, the steps for extending the Lagrangian via polynomial-bound constraints are the following:
\begin{itemize}
\item If $L(\mathbf x,\mu)$ is an ST-polynomial, add only such polynomial-bound constraints that do not
introduce new monomials;
\item If $L(\mathbf x,\mu)$ is not an ST-polynomial, add polynomial-bound constraints with exponents as in
Theorem~\ref{the:OPexponent} and discard all other vertices of the Newton polytope.
\end{itemize}


\section{Computational Results}\label{sec:computational}


In this section, we evaluate the effect of polynomial-bound constraints on bounds yielded by SONC
relaxations, in particular the numbers of instances for which a SONC relaxation of the root node
yields a finite dual bound.
Then we compare three choices of exponents of polynomial-bound constraints, and finally, we assess
the impact of SONC relaxations on the performance of a spatial branch-and-bound algorithm.

To this end, we run SCIP with standard SONC relaxations, with SONC relaxations enhanced by
polynomial-bound constraints, and default SCIP.
We distinguish the following termination statuses of the SONC relaxator:
\begin{itemize}
\item \textit{optimal}: the relaxator successfully computed a finite dual bound.

\item\textit{infeasible}: infeasibility of the node subproblem was detected.

\item \textit{unsolvable}: no valid decomposition of the Lagrangian into circuit polynomials could
be computed, or none of the optimization problems \eqref{eqn:GPproblem} and \eqref{eqn:GPproblemtwo}
was a GP, or ECOS failed to compute a solution of the GP due to a numerical error.

\item \textit{interrupted}: the relaxator was interrupted due to a time limit.

\item \textit{did not run}: the relaxator was not called
since SCIP trivially solved the instance.
\end{itemize}

We ran the tests on a development version of SCIP on a cluster of 2.60 GHz Intel Xeon E5-2660
processors with 128 GB memory per node.
The time limit was 3600 seconds and the optimality gap tolerance was $0.01\%$.
The relaxation handler was called with a frequency of 10, i.e. the relaxator was executed on
subproblems in depth levels that are multiples of 10 of the branch and bound tree.
To solve the GP, POEM called the solver ECOS via CVXPY.

We use the same test set that was used by González-Rodríguez et al.~\cite{gonzalez2022computational}.
The test set contains $349$ instances chosen from two different sources.
$180$ instances come from the set of polynomial instances randomly generated by Dalkiran and
Sherali~\cite{DalkiranSherali}, and $169$ polynomial programming instances with continuous
variables were chosen from MINLPLib~\cite{minlplib}.

In all instances, variables are bounded in the original problem formulation. 
The only exception is the variable that is introduced in order to rewrite a nonlinear objective
function as a linear objective function, and is often unbounded.
However, the relaxator removes this variable from the formulation when the linear objective is
transformed back into nonlinear form, as described in Section~\ref{sec:implementation}.





\subsection{Finiteness of Bounds}

We compare the performance of standard SONC relaxations to the performance of SONC relaxations
with added polynomial-bound constraints of the form
$$x_i^{\alpha_i} \leq \max \{\vert l_i \vert, \vert u_i \vert \}^{\alpha_i},$$
where $\alpha_i = 2 \cdot \max\{\beta_i ~\rvert~ \beta \in \Delta(L) \} + 4$.
The choice of the exponent will be discussed in more detail in the next subsection.

Table~\ref{tab:boundsYesNo} compares results produced when using the relaxation enhanced with
polynomial-bound constraints with the results produced when using the standard SONC relaxation.
It reports the statuses of the relaxator in the root node.

\begin{table}[ht]
\begin{tabular*}{\linewidth}{@{\extracolsep{\fill}}cccccc}
\sphline
\it Test Run &  \multicolumn{4}{c}{\it Solution Status of the Relaxation}  \\ 
& \it optimal & \it infeasible & \it unsolvable & \it interrupted & \it did not run   \\
\sphline
Standard SONC & 9 & 4 & 329 & 2 & 5 \\
SONC + PB & 330 & 0 & 0 & 14 & 5 \\
\sphline
\end{tabular*}
\smallskip
\caption[Comparison of the Relaxator with and without Bounds]{Effect of polynomial-bound (PB) constraints on relaxator status in the root node.}
\label{tab:boundsYesNo}
\end{table}

For most instances, the standard SONC relaxation fails to find a lower bound.
This is due to the fact that for most instances the Lagrangian function does not have a Newton
polytope with even vertices. 
Adding the polynomial-bound constraints increased the number of instances where a finite lower
bound was found from 9 to 330, as the bound constraints always provide a valid cover.

The relaxator, however, no longer detected infeasibility on the 4 instances that the standard SONC
relaxation identified as infeasible.
This might have been caused either by the change in the constant term resulting from the addition
of polynomial-bound constraints, or the relaxator not using the monomial squares that, after the
addition of polynomial-bound constraints, lie in the interior of the Newton polytope.

For 12 more instances, the relaxator was interrupted due to a time limit.
The reason for this is that with polynomial-bound constraints, it attempted to solve more GPs than
the version without the polynomial-bound constraints, since the latter often terminated after
establishing that the Newton polytope had non-monomial-square vertices.

\subsection{Choice of Exponents in Polynomial-Bound Constraints}

This section analyses the effect that the choice of exponent matrix $A' \in \mathbb{N}^{n \times n}$
in polynomial-bound constraints has on performance and dual bound quality.
As before, the matrices are diagonal, and we compare the following values of the diagonal entries:
\begin{align}
\alpha^{(n0)}_{ii} &= (n+(n \mod 2))\cdot \max\{\beta_i ~\rvert~ \beta \in \Delta(L) \},\\
\alpha^{(n4)}_{ii} &= (n + (n \mod 2)) \cdot \max\{\beta_i ~\rvert~ \beta \in \Delta(L) \} + 4,\\
\alpha^{(4)}_{ii} &= 2 \cdot \max\{\beta_i ~\rvert~ \beta \in \Delta(L) \} + 4,
\end{align}
for $i=1,\ldots,n$ and denote the respective matrices by $A^{(n0)}$, $A^{(n4)}$ and $A^{(4)}$.

The matrix $A^{(n0)}$ defines the minimal exponents that satisfy the conditions of
Theorem~\ref{the:OPexponent}, thus guaranteeing a valid cover.
However, for some instances, using $A^{(n0)}$ creates degenerate points, that is, inner points that
lie on a face of the Newton polytope.
A known issue in POEM can cause the relaxator to terminate without successfully finding a dual
bound, even though there exists a valid cover.
To avoid this issue, we increase the value of the nonzero component by $4$, which corresponds to
the matrix $A^{(n4)}$.
If the exponents defined by $A^{(n4)}$ are vertices of the Newton polytope, then no inner points
are degenerate.
This is the case for any instance where the Lagrangian function is not an ST-polynomial.

Both with $A^{(n0)}$ and $A^{(n4)}$, the degree increases linearly with the number of
variables.
Although the complexity of SONC bound computation does not depend on the degree, very high
degrees can lead to numerical issues, particularly when variable bounds have large absolute values.
This motivates our use of the values $A^{(4)}$.
Although this choice does not satisfy the conditions of Theorem~\ref{the:OPexponent} and thus does not in general
guarantee a valid cover, it is independent of $n$ and leads to a much smaller increase in
the degree.
For $n=1,2$, the matrices $A^{(n4)}$ and $A^{(4)}$ are equivalent, and $A^{(4)}$
guarantees a valid cover for each non-monomial-square term.
For $n>2$, if the inner terms are close to the axes, that is, the monomials are sparse, then
$A^{(4)}$ is still likely to provide a valid cover. 

Table~\ref{tab:exponents} reports the numbers of instances where the root node SONC relaxation
terminated with a given status for a given exponent choice.
The exponent matrix $A^{(n0)}$ leads to $13$ instances where the SONC relaxation was unsolvable due
to degenerate points, or by numerical issues caused by the value
$\max \{\vert l_i \vert, \vert u_i \vert \}^{\alpha_{ii}}$ becoming too large.
With $A^{(n4)}$, there are $9$ unsolvable instances due to such numerical issues.
The use of $A^{(4)}$ results in polynomial-bound constraints of lower degrees and alleviates these
numerical issues while still providing a valid cover for all instances in our test set.
Further, the use of exponents $A^{(4)}$ results in the smallest number of timeouts,
leading to the SONC relaxation producing valid finite dual bounds for the largest number of
instances when $A^{(4)}$ is used.

\begin{center}
\begin{table}[h]
\begin{tabular*}{\textwidth}{@{\;\;\extracolsep{\fill}}cccccc}
\toprule
Exponents &  \multicolumn{5}{c}{Solution Status}  \\
& \it optimal & \it infeasible & \it unsolvable & \it interrupted & \it did not run   \\
\midrule
$A^{(n0)}$ & 308 & 0 & 13 & 23 & 5 \\
$A^{(n4)}$ & 314 & 0 & 9 & 21 & 5 \\ 
$A^{(4)}$ & 330 & 0 & 0 & 14 & 5 \\
\bottomrule
\end{tabular*}
\smallskip
\caption[Comparison of Statuses for Different Exponents]{Comparison of solution statuses for different exponent choices.} \label{tab:exponents}
\end{table}
\end{center}


There is only a small difference in the root node dual bounds and the run times when using
different exponents.
For all but three instances, the relative differences between absolute values of the bound are
below 1\%.
On the remaining instances, there is no clear best choice.

Table~\ref{tab:exptime} shows the shifted geometric means of times and the numbers of instances
where a given exponent yielded the best time.
Given time $t$ for a given exponent and sorted times corresponding to the other two exponents, $t_1
\leq t_2$, the time $t$ is considered to be best if $t/t_1 \leq 1.01$ and $t/t_2 \leq 0.9$.

$A^{(n4)}$ leads to the fastest performance, followed by $A^{(n0)}$ and $A^{(4)}$.
However, the difference is small, and since the relaxator yields a valid dual bound  for more
instances when using $A^{(4)}$, we employ this as the default choice in the rest of our evaluations. 
\begin{table}[h]
\begin{tabular*}{\textwidth}{@{\;\;\extracolsep{\fill}}ccc}
\toprule
exponent & mean time & number of best   \\ \midrule
$A^{(n0)}$ &  122.51 & 25 \\
$A^{(n4)}$ & 120.29 & 26\\ 
$A^{(4)}$ & 123.04 & 20 \\
\bottomrule
\end{tabular*}
\smallskip
\caption[Comparison of Run Time with Different Exponents]{Comparison of run time with different exponent choices.} \label{tab:expoSolTime}
\label{tab:exptime}
\end{table}

\subsection{Effect of SONC Relaxations on Branch-and-Bound Performance}

Table~\ref{tab:SCIP-performance} compares the overall performance of default SCIP with the
performance of SCIP with the SONC relaxator.
It shows the geometric means of the solving time (shift 1 second) and nodes (shift 100 nodes), and
the numbers of instances where the LP relaxation (shown in the row ``Off'') or the SONC relaxation
(shown in the row ``On'') yielded the better dual bound in the root node.
On the remaining 38 instances, there was no difference between the LP and SONC dual bounds.

Overall, SONC relaxations are currently not competitive with the LP relaxations employed by default.
The geometric mean of the run time increases drastically when SONC relaxator is enabled, and SONC bounds are
stronger than LP bounds on only few instances.
The increase in the number of nodes is considerable smaller than the slowdown, which indicates that
the high computational cost of SONC relaxations compared to LP relaxations and the lack of warm
starting in the SONC relaxator are the main factors leading to the slowdown.
Further, the current lack of presolving, reduction and branching methods aimed at improving the
bounds yielded by SONC relaxations puts them at a disadvantage when compared to LP relaxations.

\begin{table}[h]
\begin{tabular*}{\linewidth}{@{\;\;\extracolsep{\fill}}cccc}
\toprule
Relaxator & Time & Nodes & Bound improvements \\
\midrule
Off & 25.29  & 27.04 & 305 \\
On  & 123.04 & 34.60 & 6 \\
\bottomrule
\end{tabular*}
\smallskip
\caption{Solver performance with and without the relaxator.}
\label{tab:SCIP-performance}
\end{table}

Table~\ref{tab:imprInst} provides a more detailed analysis for the 6 instances where enabling the
SONC relaxator lead to an improvement in the root node dual bound.
For two more instances, the SONC relaxator produced a tighter bound than the LP relaxation in
non-root nodes, but these instances are not included in the table.
For each instance in Table~\ref{tab:imprInst}, we report the LP and SONC dual bound, the primal
bound provided on the MINLPLib website\footnote{\url{https://www.minlplib.org}}, the percentage of
absolute gap closed, as well as information on the instance itself such as the highest polynomial
degree in the instance and the numbers of variables and constraints, disregarding the objective
variable and the objective constraint that are introduced in order to move the nonlinear part of
the objective into a constraint.

Despite the small number of such instances, the improvement in the root node dual bound is
considerable for each instance.
All instances but one are unconstrained optimization problems involving high polynomial degrees,
and on these instances SONC relaxations close between 76.1 and 91.8\% of the gap.
This is most likely due to the SONC approach for unconstrained optimization being more mature
than the SONC approach for constrained optimization, as well as the fact that LP relaxations of
problems with low-degree polynomials, especially quadratic problems, are very well developed; specialized relaxations such as SONC may be more beneficial for problems with high-degree
polynomials.

The improvement on waterund01 is smaller at 11.3\%, but
the fact that we can at all observe an improvement here, however, is already notable since waterund01 is a quadratic instance, and the
benefits of SONC relaxations are more pronounced for problems involving high degrees.

\begin{table}[h]
\begin{tabular*}{\linewidth}{@{\;\;\extracolsep{\fill}}cccccccc}
\toprule
Instance    & LP       & SONC     & Primal   & Gap closed & Degree & NVars & NConss \\
\midrule
ex4\_1\_1   & -385.308 & -97.8766 & -7.4873  & 76.1\%     & 6      & 1     & 0 \\
ex4\_1\_4   & -256     & -27      & 0.0      & 89.5\%     & 4      & 1     & 0 \\
ex4\_1\_6   & -3125    & -250     & 7.0      & 91.8\%     & 6      & 1     & 0 \\
ex4\_1\_7   & -166.636 & -44.1665 & -7.5     & 77.0\%     & 4      & 1     & 0 \\
mathopt5\_8 & -41.1406 & -8.2371  & -0.6861  & 81.3\%     & 6      & 1     & 0 \\
waterund01  & -307.557 & -262.889 & 86.8333  & 11.3\%     & 2      & 40    & 38 \\
\bottomrule
\end{tabular*}
\smallskip
\caption[Improvement of Dual Bounds in the Root Node by Relaxator]{Instances with root node dual bound improvements when using SONC relaxations.}
\label{tab:imprInst}
\end{table}


\section{Conclusion}\label{sec:conclusion}

In this paper, we have developed a method to guarantee that, given finite variable domains, the SONC
relaxation returns a finite dual bound.
This method enables the SONC relaxation to make use of tighter variable bounds in the
branch-and-bound tree.
Our experiments showed that for most instances in our test set, the standard SONC relaxation failed
to find a finite dual bound, whereas the SONC relaxation with added polynomial-bound constraints
found such a bound for all instances except the ones where the relaxator exceeded the time limit.
Thus, our approach drastically improves the applicability of SONC relaxations to constrained
polynomial optimization problems.

Further, we implemented an experimental algorithm that solves SONC relaxations of node subproblems
in a branch-and-bound tree.
In its current version, this algorithm is not competitive with the LP-based branch-and-bound
methods.
This is not surprising
since it lacks a full integration of SONC relaxations with the branch-and-bound
algorithm, whereas the integration of LP relaxations has been the topic of many years of research.
Moreover, the treatment of constraints and variable bounds in SONC relaxations needs further
refinement.

However, our experiments showed that our approach has potential, since for some instances, the SONC
relaxation yielded dual bounds that closed up to 91.8\% of the root node gap when compared to the LP
relaxation.
Further work would be focused on designing a SONC-based branch-and-bound algorithm where reductions
and branching decisions are aimed at strengthening SONC relaxations, and further improving the
capabilities of SONC relaxations to make use of bounds and constraints.



\section*{Acknowledgments}

The work for this article has been conducted within the Research
Campus Modal funded by the German Federal Ministry of Education and Research (BMBF grant numbers
05M14ZAM, 05M20ZBM).

\section*{Declarations}

\subsection*{Competing interests}

The authors have no competing interests to declare that are relevant to the content of this article.

\subsection*{Availability of data}

Datasets generated and analysed during the current study are available at
\url{https://github.com/scipopt/PySCIPOpt} and \url{https://www.minlplib.org/}.
The authors are in the process of making the remaining data publicly available.


\bibliographystyle{plain}
\bibliography{sonc}

\begin{thebibliography}{10}

\bibitem{adams1986tight}
Warren~P Adams and Hanif~D Sherali.
\newblock A tight linearization and an algorithm for zero-one quadratic
  programming problems.
\newblock {\em Management Science}, 32(10):1274--1290, 1986.

\bibitem{adams1990linearization}
Warren~P Adams and Hanif~D Sherali.
\newblock Linearization strategies for a class of zero-one mixed integer
  programming problems.
\newblock {\em Operations Research}, 38(2):217--226, 1990.

\bibitem{ahmadi2011algebraic}
Amir~Ali Ahmadi.
\newblock {\em Algebraic relaxations and hardness results in polynomial
  optimization and {L}yapunov analysis}.
\newblock PhD thesis, Massachusetts Institute of Technology, 2011.

\bibitem{ahmadi2019dsos}
Amir~Ali Ahmadi and Anirudha Majumdar.
\newblock {DSOS} and {SDSOS} optimization: more tractable alternatives to sum
  of squares and semidefinite optimization.
\newblock {\em SIAM Journal on Applied Algebra and Geometry}, 3(2):193--230,
  2019.

\bibitem{SCIPoptsuite80}
Ksenia Bestuzheva, Mathieu Besan{\c{c}}on, Wei-Kun Chen, Antonia Chmiela, Tim
  Donkiewicz, Jasper van Doornmalen, Leon Eifler, Oliver Gaul, Gerald Gamrath,
  Ambros Gleixner, Leona Gottwald, Christoph Graczyk, Katrin Halbig, Alexander
  Hoen, Christopher Hojny, Rolf van~der Hulst, Thorsten Koch, Marco
  L{\"u}bbecke, Stephen~J. Maher, Frederic Matter, Erik M{\"u}hmer, Benjamin
  M{\"u}ller, Marc~E. Pfetsch, Daniel Rehfeldt, Steffan Schlein, Franziska
  Schl{\"o}sser, Felipe Serrano, Yuji Shinano, Boro Sofranac, Mark Turner,
  Stefan Vigerske, Fabian Wegscheider, Philipp Wellner, Dieter Weninger, and
  Jakob Witzig.
\newblock {The SCIP Optimization Suite 8.0}.
\newblock ZIB Report 21-41, Zuse Institute Berlin, 2021.

\bibitem{boyd2007tutorial}
Stephen Boyd, Seung-Jean Kim, Lieven Vandenberghe, and Arash Hassibi.
\newblock A tutorial on geometric programming.
\newblock {\em Optimization and engineering}, 8(1):67--127, 2007.

\bibitem{BoydVan}
Stephen Boyd and Lieven Vandenberghe.
\newblock {\em Convex optimization}, volume~9, pages 457--483.
\newblock Cambridge University Press, Cambridge, 2004.

\bibitem{chandrasekaran2016relative}
Venkat Chandrasekaran and Parikshit Shah.
\newblock Relative entropy relaxations for signomial optimization.
\newblock {\em SIAM Journal on Optimization}, 26(2):1147--1173, 2016.

\bibitem{chandrasekaran2017relative}
Venkat Chandrasekaran and Parikshit Shah.
\newblock Relative entropy optimization and its applications.
\newblock {\em Mathematical Programming}, 161(1):1--32, 2017.

\bibitem{commander2008jamming}
Clayton~W Commander, Panos~M Pardalos, Valeriy Ryabchenko, Oleg Shylo, Stan
  Uryasev, and Grigoriy Zrazhevsky.
\newblock Jamming communication networks under complete uncertainty.
\newblock {\em Optimization Letters}, 2(1):53--70, 2008.

\bibitem{commander2007wireless}
Clayton~W Commander, Panos~M Pardalos, Valeriy Ryabchenko, Stan Uryasev, and
  Grigoriy Zrazhevsky.
\newblock The wireless network jamming problem.
\newblock {\em Journal of Combinatorial Optimization}, 14(4):481--498, 2007.

\bibitem{DalkiranSherali}
Evrim Dalkiran and Hanif Sherali.
\newblock Rlt-pos: Reformulation-linearization technique-based optimization
  software for solving polynomial programming problems.
\newblock {\em Mathematical Programming Computation}, 8, 02 2016.

\bibitem{diamond2016cvxpy}
Steven Diamond and Stephen Boyd.
\newblock {CVXPY}: A {P}ython-embedded modeling language for convex
  optimization.
\newblock {\em The Journal of Machine Learning Research}, 17(1):2909--2913,
  2016.

\bibitem{domahidi2013ecos}
Alexander Domahidi, Eric Chu, and Stephen Boyd.
\newblock {ECOS}: An {SOCP} solver for embedded systems.
\newblock In {\em 2013 European Control Conference (ECC)}, pages 3071--3076.
  IEEE, 2013.

\bibitem{GP}
Mareike Dressler, Sadik Iliman, and Timo de~Wolff.
\newblock An approach to constrained polynomial optimization via nonnegative
  circuit polynomials and geometric programming, 2016.

\bibitem{gonzalez2022computational}
Brais Gonz{\'a}lez-Rodr{\'\i}guez, Joaqu{\'\i}n Ossorio-Castillo, Julio
  Gonz{\'a}lez-D{\'\i}az, {\'A}ngel~M Gonz{\'a}lez-Rueda, David~R Penas, and
  Diego Rodr{\'\i}guez-Mart{\'\i}nez.
\newblock Computational advances in polynomial optimization: {RAPOS}a, a freely
  available global solver.
\newblock {\em Journal of Global Optimization}, pages 1--28, 2022.

\bibitem{henrion2009gloptipoly}
Didier Henrion, Jean-Bernard Lasserre, and Johan L{\"o}fberg.
\newblock Gloptipoly 3: moments, optimization and semidefinite programming.
\newblock {\em Optimization Methods \& Software}, 24(4-5):761--779, 2009.

\bibitem{horst2013global}
Reiner Horst and Hoang Tuy.
\newblock {\em Global optimization: Deterministic approaches}.
\newblock Springer Science \& Business Media, Berlin, 2013.

\bibitem{SONC}
Sadik Iliman and Timo de~Wolff.
\newblock Amoebas, nonnegative polynomials and sums of squares supported on
  circuits.
\newblock {\em Res. Math. Sci.}, 3:Paper No. 9, 35, 2016.

\bibitem{Lower}
Sadik Iliman and Timo de~Wolff.
\newblock Lower bounds for polynomials with simplex {N}ewton polytopes based on
  geometric programming.
\newblock {\em SIAM J. Optim.}, 26(2):1128--1146, 2016.

\bibitem{lasserre2001global}
Jean~B Lasserre.
\newblock Global optimization with polynomials and the problem of moments.
\newblock {\em SIAM Journal on optimization}, 11(3):796--817, 2001.

\bibitem{maher2016pyscipopt}
Stephen Maher, Matthias Miltenberger, Jo{\~a}o~Pedro Pedroso, Daniel Rehfeldt,
  Robert Schwarz, and Felipe Serrano.
\newblock Py{SCIPO}pt: Mathematical programming in {P}ython with the {SCIP}
  {O}ptimization {S}uite.
\newblock In {\em International Congress on Mathematical Software}, pages
  301--307. Springer, 2016.

\bibitem{majumdar2013control}
Anirudha Majumdar, Amir~Ali Ahmadi, and Russ Tedrake.
\newblock Control design along trajectories with sums of squares programming.
\newblock In {\em 2013 IEEE International Conference on Robotics and
  Automation}, pages 4054--4061. IEEE, 2013.

\bibitem{mattingley2010real}
John Mattingley and Stephen Boyd.
\newblock Real-time convex optimization in signal processing.
\newblock {\em IEEE Signal processing magazine}, 27(3):50--61, 2010.

\bibitem{minlplib}
A library of mixed-integer and continuous nonlinear programming instances.
\newblock \url{https://www.minlplib.org}, 2022-10-14.

\bibitem{murray2021signomial}
Riley Murray, Venkat Chandrasekaran, and Adam Wierman.
\newblock Signomial and polynomial optimization via relative entropy and
  partial dualization.
\newblock {\em Mathematical Programming Computation}, 13(2):257--295, 2021.

\bibitem{nesterov2000squared}
Yurii Nesterov.
\newblock Squared functional systems and optimization problems.
\newblock In {\em High performance optimization}, pages 405--440. Springer,
  Boston, 2000.

\bibitem{sostools}
Antonis Papachristodoulou, James Anderson, Giorgio Valmorbida, Stephen Prajna,
  Pete Seiler, Pablo Parrilo, Matthew Peet, and Declan Jagt.
\newblock {\em {SOSTOOLS}: Sum of squares optimization toolbox for {MATLAB}}.
\newblock \texttt{http://arxiv.org/abs/1310.4716}, 2021.
\newblock Available from \texttt{https://github.com/oxfordcontrol/SOSTOOLS}.

\bibitem{parrilo2003semidefinite}
Pablo~A Parrilo.
\newblock Semidefinite programming relaxations for semialgebraic problems.
\newblock {\em Mathematical programming}, 96(2):293--320, 2003.

\bibitem{reznick1978extremal}
Bruce Reznick.
\newblock Extremal {PSD} forms with few terms.
\newblock {\em Duke mathematical journal}, 45(2):363--374, 1978.

\bibitem{seidler2021improved}
Henning Seidler.
\newblock Improved lower bounds for global polynomial optimisation.
\newblock {\em arXiv preprint arXiv:2105.14124}, 2021.

\bibitem{seidler2018experimental}
Henning Seidler and Timo de~Wolff.
\newblock An experimental comparison of {SONC} and {SOS} certificates for
  unconstrained optimization.
\newblock {\em arXiv preprint arXiv:1808.08431}, 2018.

\bibitem{poem:software}
Henning Seidler and Timo de~Wolff.
\newblock {POEM}: Effective methods in polynomial optimization, version
  0.2.1.0(a).
\newblock \url{http://www.iaa.tu-bs.de/AppliedAlgebra/POEM/index.html}, July
  2019.

\bibitem{sherali2012reduced}
Hanif~D Sherali, Evrim Dalkiran, and Leo Liberti.
\newblock Reduced {RLT} representations for nonconvex polynomial programming
  problems.
\newblock {\em Journal of Global Optimization}, 52(3):447--469, 2012.

\bibitem{sherali1997new}
Hanif~D Sherali and Cihan~H Tuncbilek.
\newblock New reformulation linearization/convexification relaxations for
  univariate and multivariate polynomial programming problems.
\newblock {\em Operations Research Letters}, 21(1):1--9, 1997.

\bibitem{shor1987class}
Naum~Z Shor.
\newblock Class of global minimum bounds of polynomial functions.
\newblock {\em Cybernetics}, 23(6):731--734, 1987.

\bibitem{sturmfels2002solving}
Bernd Sturmfels.
\newblock {\em Solving systems of polynomial equations}.
\newblock Number~97. American Mathematical Society, Rhode Island, 2002.

\bibitem{waki2006sums}
Hayato Waki, Sunyoung Kim, Masakazu Kojima, and Masakazu Muramatsu.
\newblock Sums of squares and semidefinite program relaxations for polynomial
  optimization problems with structured sparsity.
\newblock {\em SIAM Journal on Optimization}, 17(1):218--242, 2006.

\bibitem{zheng2021sum}
Yang Zheng and Giovanni Fantuzzi.
\newblock Sum-of-squares chordal decomposition of polynomial matrix
  inequalities.
\newblock {\em Mathematical Programming}, pages 1--38, 2021.

\end{thebibliography}


\end{document}